\documentclass[a4paper,10pt,leqno]{amsart}
\usepackage{amsmath,amsfonts,amssymb,amsthm,epsfig,epstopdf,url,array}
\usepackage{amsthm}
\usepackage{mathrsfs}
\usepackage{epsfig}
\usepackage{graphicx}

\newtheorem{thm}{Theorem}[section]
\newtheorem{prop}[thm]{Proposition}

\newtheorem{lem}[thm]{Lemma}
\newtheorem{rmk}[thm]{Remark}

\numberwithin{equation}{section}

\newcommand{\z}{{\mathbb{Z}}}

\newcommand{\nequiv}{\not\equiv}

\newcommand{\rank}{\operatorname{rank}}

\newcommand{\floor}[1]{\left\lfloor #1 \right\rfloor}
\newcommand{\ceil}[1]{\left\lceil #1 \right\rceil}

\title[The rank of universal $m$-gonal forms]{The rank of universal $m$-gonal forms}

\author{Byeong Moon Kim and Dayoon Park}
\address{Department of Mathematics, Kangnung National University, Kangnung, 210-702, Korea}
\email{kbm@kangnung.ac.kr}
\thanks{}

\address{Department of Mathematics, The University of Hong Kong, Hong Kong}
\email{pdy1016@hku.hk}
\thanks{}

\begin{document}
\maketitle

\begin{abstract}

In this article, we consider the rank of universal $m$-gonal forms for all sufficiently large $m$.
Especially, we determine the minimal rank of universal $m$-gonal form and the maximal rank of kinds of proper universal $m$-gonal form.
\end{abstract}

\section{introduction}
For $m \ge 3$, the {\it $m$-gonal number} is described as the total number of dots to constitute a regular $m$-gon.
One may easily induce a formula
\begin{equation} \label{m number}
P_m(x)=\frac{m-2}{2}x^2-\frac{m-4}{2}x
\end{equation}
for the total number of dots to constitute a regular $m$-gon with $x$ dots for each side.
We especially call the $m$-gonal number of (\ref{m number}) as {\it x-th $m$-gonal number}.
For a long time, the representation of positive integers by a sum of $m$-gonal numbers has been one of popular subjects in the field of number theory.
In 17-th Century, Fermat conjectured that every positive integer may be written as at most $m$ $m$-gonal numbers.
And Lagrange and Gauss resolved his conjecture for $m=4$ and $m-3$ in 1770 and 1796, respectively. 
Finally Cauchy presented a proof for his conjecture for all $m \ge 3$ in 1813.

By definition of $m$-gonal number, only positive integer would be admitted to $x$ in (\ref{m number}). 
On the other hand, by considering $P_m(x)$ in (\ref{m number}) with $x \in \z_{\le 0}$ as $m$-gonal number too, we may generalize the $m$-gonal number.
As a general version of Fermat's Conjecture, very recently, the second author \cite{det} completed the minimal $\ell_m$ for which every positive integer may be written as $\ell_m$ (generalized) $m$-gonal numbers for all $m \ge 3$ as
$$ \ell_m=
\begin{cases} m-4 & \text{ if } m \ge 9\\
3& \text{ if } m \in \{3,5,6\}\\
4& \text{ if } m \in \{4,7,8\}. \\
\end{cases}$$

To consider the representation of positive integers by $m$-gonal numbers more generally, we may think about the weighted sum of $m$-gonal numbers
\begin{equation}\label{m form}
F_m(\mathbf x)=a_1P_m(x_1)+\cdots+a_nP_m(x_n)
\end{equation}
where $a_i \in \mathbb N$ admitting $x_i \in \z$.
We call $F_m(\mathbf x)$ in (\ref{m form}) as {\it $m$-gonal form}.
In this paper, without loss of generality, we assume that $a_1 \le \cdots \le a_n$.
If the diophantine equation $$F_m(\mathbf x)=N$$ has an integer solution $\mathbf x \in \z^n$ for $N \in \mathbb N$, then we say that the $m$-gonal form $F_m(\mathbf x)$ {\it represents $N$}.
Naturally, an $m$-gonal form which represents every positive integer may be paid attention.
We call an $m$-gonal form which represents every positive integer as {\it universal}.
Every universal $3$-gonal forms was classified by Liouville in $19$-th Century.
And Ramanujam showd all universal $4$-gonal forms, in fact, one of them is not actually universal. 

In general, to determine whether a form is universal is not easy.
On the other hand, Conway and Schneeberger announced amazingly simple criterion to determine universality of a quadratic form.
The result which is well known as {\it $15$-Theorem} states that the representability of positive integers up to only $15$ by a quadratic form characterize the universality of the quadratic form.
Kane and Liu \cite{BJ} claimed that such a finiteness theorem holds for any $m$-gonal form too.
On the other words, there is (unique and minimal) $\gamma_m$ for which if an $m$-gonal form represents every positive integer up to $\gamma_m$, then the $m$-gonal form is universal.
They \cite{BJ} also questioned about the growth of $\gamma_m$ (which is asymptotically increasing) and showed that
$$m-4 \le \gamma_m \ll m^{7+\epsilon}.$$
The authors \cite{KP'} obtained the optimal growth of $\gamma_m$ which is exactly linear on $m$ by showing Theorem \ref{C}.
\begin{thm} \label{C}
For $m \ge 3$, there exists an absolute constant $C$ such that $\gamma_m \le C(m-2)$.
\end{thm}
\begin{proof}
See \cite{KP'}.
\end{proof}

On the other hand, Bhargava \cite{CS} suggested a simple proof for $15$-Theorem by introducing {\it escalator tree} of quadratic form.
Following the Bhargava's escalating method, we may consider an escalator tree of $m$-gonal form.
For a non universal form, we call the minimal integer which is not represented by the form as the {\it truant} of the form.
We call a super form of a non-universal form which represents the truant of the non-universal form as {\it escalator} of the form.
The escalator tree is a rooted tree consisting of $m$-gonal forms having the root $\emptyset$.
If a node $\sum_{i=1}^ka_iP_m(x_i)$ of the tree is not universal, then the node spreads branches by taking its children all of its escalotors $\sum_{i=1}^{k+1}a_iP_m(x_i)$ with $a_k \le a_{k+1}$.
And if a node does not have truant (i.e., the node is universal), then the node would be a leaf of the tree. 
In other words, an universal $m$-gonal form would only appear on the leaves of the escalating tree, all leaves of the tree are universal and all of the universal $m$-gonal forms on the leaves of the tree are kind of proper universal $m$-gonal forms in the sense that without its last component the universality is broken (i.e., its parents are not universal).
The $\gamma_m$ would be exactly the maximal truant of node of escalator tree of $m$-gonal form.
In effect, in \cite{CS}, Bhargava found the maximal truant of node of escalator tree of quadratic form.
Meanwhile, one may catch that an $m$-gonal form would contain at least one of leaves of escalator tree as its subform.
In this paper, we treat the escalator tree of $m$-gonal form for sufficiently large $m$.
Especially, we determine the minimal rank $r_m$ and maximal rank $R_m$ of leaf of the escalator tree and show there is a leaf of rank $n$ for any $n \in [r_m,R_m]$ for all sufficiently large $m$.
The $r_m$ would be indeed the minimal rank of universal $m$-gonal form.
So we may also obtain the minimal rank of universal $m$-gonal form for sufficiently large $m$.
And any $m$-gonal form would contain a universal $m$-gonal form of rank less than or equal to $R_m$.
Overall, the results would provide the answers of most part on the rank of universal $m$-gonal forms.
In this paper, we basically use the arithmetic theory of quadratic form.
Any unexplained notation and terminology can be found in \cite{O1} and \cite{O}.

\vskip 0.5em

The paper is organized as follows.
In section 2, we introduce escalator tree more concretely and our results.
In Section 3, we determine $r_m$ and see a leaf of rank $r_m$.
In Section 4, we determine $R_m$ and see a leaf of rank $R_m$.
\section{preliminary}

Following the Guy's argument \cite{G}, since the smallest (generalized) $m$-gonal number is $m-3$ except $0$ and $1$, we may yield that every node $\sum _{i=1}^ka_iP_m(x_i)$ of the escalator tree must satisfy the following conditions
\begin{equation}\label{coe}\begin{cases}a_1=1 &  \\ a_{i+1}\le a_1+\cdots +a_i+1 & \text{if } a_1+\cdots +a_i <m-4\end{cases}\end{equation}
because the truant of a node $\sum_{i=1}^ka_iP_m(x_i)$ with $a_1+\cdots +a_k <m-4$ would be $$a_1+\cdots + a_k +1.$$
From (\ref{coe}), we may have that 
\begin{equation}\label{2^i}a_{i+1} \le 2^i \quad \text{ when }a_1+\cdots +a_i <m-4.\end{equation}
By the Guy's argument \cite{G} again, every leaf $\sum _{i=1}^n a_iP_m(x_i)$ which is universal must have 
\begin{equation}\label{uni}a_1+\cdots +a_n \ge m-4\end{equation}
since otherwise the integers from $a_1+\cdots+a_n+1$ to $m-4$ cannot be represented by the (universal) leaf, which is a contradiction.
From (\ref{2^i}) and (\ref{uni}), we may obtain that every leaf in the escalator tree would have the rank greater than or equal to $\ceil{\log_2(m-3)}$, i.e., $$\ceil{\log_2(m-3)} \le r_m.$$
And following the Theorem \ref{C}, a truant could not exceed $C(m-2)$.
So we may clearly obtain that $$R_m \le C(m-2).$$
Throughtout this paper, we exactly determine the $r_m$ and $R_m$ for all $m$ sufficiently large.
In Chapter 3, we prove the following theorem. 
\begin{thm} \label{min}For $m>2\left(\left(2C+\frac{1}{4}\right)^{\frac{1}{4}}+\sqrt{2}\right)^2$,
\begin{equation}\label{thmrm}r_m=\begin{cases}\ceil{\log_2(m-3)}+1& \text{ when } -3 \le 2^{\ceil{\log_2(m-3)}}-m\le 1\\
\ceil{\log_2(m-3)} & \text{ when } \ \quad 2 \le 2^{\ceil{\log_2(m-3)}}-m.
\end{cases}\end{equation}
\end{thm}

\vskip 0.5em

\begin{rmk}
Furthermore, in the proof of Therem \ref{min}, we claim that for $m>2\left(\left(2C+\frac{1}{4}\right)^{\frac{1}{4}}+\sqrt{2}\right)^2$, 
$$P_m(x_1)+2P_m(x_2)+\cdots +2^{r_m-1}P_m(x_{r_m})$$ is universal $m$-gonal form of the minimal $\rank r_m$.
Actually, the authors guess that (\ref{thmrm}) in Theorem \ref{min} holds for much smaller $m$'s too.
But we would not remove the restriction $m>2\left(\left(2C+\frac{1}{4}\right)^{\frac{1}{4}}+\sqrt{2}\right)^2$ in Theorem \ref{min} with the arguments in Chapter 3 even though one could take slightly smller lower bound instead of $2\left(\left(2C+\frac{1}{4}\right)^{\frac{1}{4}}+\sqrt{2}\right)^2$ through more careful care.
\end{rmk}

\vskip 0.5em

In Chapter 4, we determine $R_m$ by claming following theorem.

\begin{thm}\label{main}
For $m>6C^2(C+1)$,
$$R_m=\begin{cases}m-2 & \text{ when } m\nequiv 2 \pmod{3}\\m-3 & \text{ when } m\equiv 2 \pmod{3}.\end{cases}$$
\end{thm}

\vskip 0.5em

\begin{rmk} \label{rmk R}
For $m>6C^2(C+1)$, there would be exactly $3m-12$ and $3m-14$ leaves of  the $\rank \ R_m$ when $m\equiv0$ and $m\not\equiv0 \pmod{3}$, respectively. 
In particular, we would characterize all of the leaves of $\rank R_m$ as follows.
\begin{itemize}
\item[1.]When $m\equiv 0 \pmod{3}$, all of the leaves of $\rank \ R_m$ are $$P_m(x_1)+P_m(x_2)+\sum \limits_{k=3}^{m-3}3P_m(x_k)+a_{m-2}P_m(x_{m-2})$$ where $3 \le a_{m-2} \le 3m-10.$\\
\item[2.]When $m\equiv 1 \pmod{3}$, all of the leaves of $\rank \ R_m$ are $$P_m(x_1)+P_m(x_2)+\sum \limits_{k=3}^{m-3}3P_m(x_k)+a_{m-2}P_m(x_{m-2})$$ where $3 \le a_{m-2} \le 3m-12.$\\
\item[3.]When $m\equiv 2 \pmod{3}$, all of the leaves of $\rank \ R_m$ are $$P_m(x_1)+2P_m(x_2)+\sum \limits_{k=3}^{m-4}3P_m(x_k)+a_{m-3}P_m(x_{m-3})$$ where $3 \le a_{m-3} \le 3m-12.$
\end{itemize}
\end{rmk}

\vskip 1em

\section{A universal $m$-gonal form of the minimal rank : The most fastly escalated universal $m$-gonal form}
We may see that a node $\sum \limits_{i=1}^ka_iP_m(x_i)$ with $a_1+\cdots +a_k <m-4$ of the esacalator tree have the truant $a_1+\cdots+a_k+1$.
So a leaf $\sum \limits_{i=1}^n a_iP_m(x_i)$(which is universal) having no truant of the tree should satisfy the followings 
\begin{equation}\label{3.1}
\begin{cases}a_1=1 & \\ 
a_{i+1}\le a_1+\cdots +a_i+1 & \text{if } a_1+\cdots +a_i <m-4 \\ 
a_1+\cdots +a_n \ge m-4. & \end{cases}\end{equation}
From the first and second conditions in (\ref{3.1}), we may obtain that for any leaf $\sum \limits_{i=1}^n a_iP_m(x_i)$, 
\begin{equation}\label{3.2}a_{k+1}\le 2^k\end{equation}
holds whenever $a_1+\cdots +a_k<m-4$.
And then with the third condition in (\ref{3.1}) and (\ref{3.2}), we may induce that every leaf $\sum \limits_{i=1}^n a_iP_m(x_i)$ has the rank $n$ greater than or equal to $\ceil{\log_2(m-3)}$.

Note that there is a node 
\begin{equation}\label{mf}\sum \limits_{i=1}^k2^{i-1}P_m(x_i)\end{equation} 
where $k=\ceil{\log_2(m-3)}$ of the tree.
The $m$-gonal form (\ref{mf}) would be one of the most fastly escalated $m$-gonal forms to represent up to $m-4$.
For sufficiently large $m>2\left(\left(2C+\frac{1}{4}\right)^{\frac{1}{4}}+\sqrt{2}\right)^2$ with $2 \le 2^{\ceil{\log_2(m-3)}}-m$, by showing that the $m$-gonal form (\ref{mf}) is universal, we claim that the $m$-gonal form is a leaf of the minimal $\rank r_m$, yielding the $m$-gonal form is indeed a universal $m$-gonal form of the minimal rank.

\begin{lem} \label{rm1}
For $m>2\left(\left(2C+\frac{1}{4}\right)^{\frac{1}{4}}+\sqrt{2}\right)^2$ with $2 \le 2^{\ceil{\log_2(m-3)}}-m$, the $m$-gonal form 
\begin{equation}\label{minm}
F_m(\mathbf x)=P_m(x_1)+2P_m(x_2)+\cdots+2^{n-1}P_m(x_n)
\end{equation}
where $n=\ceil{\log_2(m-3)}$ is universal.
\end{lem}
\begin{proof}
In virtue of the Theorem \ref{C}, it may be enough to show that $F_m(\mathbf x)$ represents every positive integer up to only $C(m-2).$
Throughout this proof, we write the integers in $[1,C(m-2)]$ as $$A(m-2)+B$$ where $0\le A \le C$ and $0 \le B \le m-3$.

Note that 
\begin{align*}
F_m(\mathbf x) = & (m-2)\{ P_3(x_1-1)+2P_3(x_2-1)+4P_3(x_3-1)+8P_3(x_4-1)\} \\
& + (x_1+2x_2+4x_3+8x_4)+16P_m(x_5)+\cdots+2^{n-1}P_m(x_n).\end{align*}
For a non-negative integer $A$, let $x(A)$ be the largest positive integer satisfying $$P_3(x(A)-1) \le A ,$$
i.e., the integer in the interval $(\sqrt{2A+\frac{1}{4}}-\frac{1}{2}, \sqrt{2A+\frac{1}{4}}+\frac{1}{2}]$.
On the other hand, there would be exactly one $$y_1(A,B) \in \{x(A), x(A)-1, x(A)-2, x(A)-3\}$$ satisfying
$$\begin{cases}
P_3(y_1(A,B)-1) \equiv A \pmod{2} \\
y_1(A,B) \equiv B  \pmod{2}.
\end{cases}$$
Since the ternary triangular form $$P_3(x)+2P_3(y)+4P_3(z)$$ is universal, the even integer $A-P_3(y_1(A,B)-1)$ may be written as
$$2P_3(y_1(A,B)-1)+4P_3(y_2(A,B)-1)+8P_3(y_3(A,B)-1)$$
for some $y_i(A,B) \in \z$, i.e., $$A=P_3(y_1(A,B)-1)+2P_3(y_1(A,B)-1)+4P_3(y_2(A,B)-1)+8P_3(y_3(A,B)-1).$$
Beside that since $P_3(x)=P_3(-x-1)$, if it is necessary, by changing $y_i(A,B)-1$ to $-y_i(A,B)$, we may assume that 
\begin{equation}\label{y}
\begin{cases} A=P_3(y_1(A,B)-1)+\cdots +8P_3(y_4(A,B)-1) \\B \equiv y_1(A,B)+2y_2(A,B)+4y_3(A,B)+8y_4(A,B) \pmod{16}. \end{cases}
\end{equation}
Since the integer $y_1(A,B)$ in (\ref{y}) is in $[\sqrt{2A+\frac{1}{4}}-\frac{7}{2}, \sqrt{2A+\frac{1}{4}}+\frac{1}{2}]$ we may get that 
\begin{align*}
0 & \le A-P_3(y_1(A,B)-1)  \\
 & =2P_3(y_2(A,B)-1)+4P_3(y_3(A,B)-1)+8P_3(y_4(A,B)-1)<4\sqrt{2A+\frac{1}{4}}-8.\\
\end{align*}
By arranging the above inequality, we may obtain
\begin{equation}\label{sqrt}\left(y_2(A,B)-\frac{1}{2}\right)^2+2\left(y_3(A,B)-\frac{1}{2} \right)^2+4\left(y_4(A,B)-\frac{1}{2} \right)^2<4\sqrt{2A+\frac{1}{4}}-\frac{25}{4}\end{equation}
and through a basic calculation, we may get that $$|2y_2(A,B)+4y_3(A,B)+8y_4(A,B)|<14\sqrt{\frac{4\sqrt{2A+\frac{1}{4}}-\frac{25}{4}}{73}}.$$
So for $y_1(A,B) \in [\sqrt{2A+\frac{1}{4}}-\frac{7}{2}, \sqrt{2A+\frac{1}{4}}+\frac{1}{2}]$, we may get that such the above $y_i(A,B)$ where $160 \le A \le C$ satisfy
\begin{align*}
0 & < \sqrt{2A+\frac{1}{4}}-14\sqrt{\frac{4\sqrt{2A+\frac{1}{4}}-\frac{25}{4}}{73}}-\frac{7}{2}\\
 & < y_1(A,B)+2y_2(A,B)+4y_3(A,B)+8y_4(A,B) \\ & <\sqrt{2A+\frac{1}{4}}+14\sqrt{\frac{4\sqrt{2A+\frac{1}{4}}-\frac{25}{4}}{73}}+\frac{1}{2} < \frac{m-2}{2}.
\end{align*}
Through similar processings with the above we may obtain  
$$\begin{cases} A=P_3(z_1(A,B)-1)+\cdots+8P_3(z_4(A,B)-1) \\ B \equiv z_1(A,B)+2z_2(A,B)+4z_3(A,B)+8z_4(A,B)  \pmod{16}\end{cases}$$
hold for some $z_1(A,B) \in \{-x(A)+1, -x(A)+2,-x(A)+3,-x(A)+4\}$ and $z_i(A,B) \in \z$ for $i=2,3,4$.
And in this case, we may get that such the above $z_i(A,B)$ with $160 \le A \le C$ satisfy 
\begin{align*}-\frac{m-2}{2} & < -\sqrt{2A-\frac{1}{4}}-14\sqrt{\frac{4\sqrt{2A+\frac{1}{4}}-\frac{25}{4}}{73}}+\frac{1}{2}  \\ & \le z_1(A,B)+2z_2(A,B)+4z_3(A,B)+8z_4(A,B) \\ 
& <-\sqrt{2A+\frac{1}{4}}+14\sqrt{\frac{4\sqrt{2A+\frac{1}{4}}-\frac{25}{4}}{73}}+\frac{9}{2}<0.\end{align*}
And then one may easily see that for each integer $A(m-2)+B \in [160(m-2), C(m-2)]$, $$\text{either } \ (x_i(A,B))=(y_i(A,B)) \ \text{ or } \ (x_i(A,B))=(z_i(A,B))$$ satisfies 
$$0 \le A(m-2)+B-\{P_m(x_1(A,B))+\cdots
+8P_m(x_4(A,B))\} \le m-11$$
with $$A(m-2)+B-\{P_m(x_1(A,B))+\cdots
+8P_m(x_4(A,B))\}\equiv 0 \pmod{16}.$$
On the other hand, remaining $16P_m(x_5)+\cdots +2^{n-1}P_m(x_n)$ may represent all the multiples of $16$ up to $m-11(\le 2^n-16)$
by taking $P_m(x_i) \in \{0,1\}$ for all $5\le i \le n$ which yields that $A(m-2)+B$ may be represented by $F_m(\mathbf x)$ as follows 
$$P_m(x_1(A,B))+\cdots+8P_m(x_4(A,B))+16P_m(x_5)+\cdots +2^{n-1}P_m(x_n)$$
for some $(x_5,\cdots, x_n) \in \{0,1\}^{n-4}$.
Until now, we showed that $$F_m(\mathbf x)=P_m(x_1)+2P_m(x_2)+\cdots+2^{n-1}P_m(x_n)$$ represents every positive integer in $[160(m-2),C(m-2)].$

In the remaining of this proof, we show that $F_m(\mathbf x)$ represents every positive integer in $[1, 160(m-2)]$.
Through direct calcuations (the authors used python), we may obtain that for each $(A,r_B) \in \z \times \z/8\z$ with $0 \le A \le 160$, there are integer solutions $(x_1,x_2,x_3)\in \z^3$ for both of
\begin{equation} \label{pos}
\begin{cases}P_3(x_1-1)+2P_3(x_2-1)+4P_3(x_3-1)=A \\ 
x_1+2x_2+4x_3 \equiv r_B \pmod{8} \\ 
0 \le x_1+2x_2+4x_3 <100 \ll \frac{m-2}{2}
\end{cases}
\end{equation} 
and 
\begin{equation} \label{neg}
\begin{cases}
P_3(x_1-1)+2P_3(x_2-1)+4P_3(x_3-1)=A \\ 
x_1+2x_2+4x_3 \equiv r_B \pmod{8} \\ 
-\frac{m-2}{2}\ll -100<x_1+2x_2+4x_3 \le 0 
\end{cases}
\end{equation} 
respectively except the pairs $(A,r_B)$ in $S^+ \cup S^- (\subset \z \times \z/8\z)$ where 
\begin{equation}\label{(a,r)pos}
\begin{array} {lllllll}
S^+:= &\{(0,0),  \!&(0,1),  \!&(0,2),\!& (0,3), \!& (0,4),  \!& (0,5),  \\
\!\!\!\!&\  (0,6),  \!&(0,7),  \!&  (1,0),  \!&(1,1),  \!&(1,2),\!& (1,3),  \\
\!\!\!\!&\   (1,4),  \!& (1,5), \!& (1,6),  \!&(2,0),  \!&(2,1),\!& (2,2),\\
\!\!\!\!&\   (2,3),  \!& (2,4),\!& (2,5),  \!&(3,1),  \!&(3,2),\!& (3,3), \\
\!\!\!\!&\  (3,4),  \!& (3,7), \!& (4,0),  \!&(4,1),  \!&(4,2),\!& (4,3), \\
\!\!\!\!&\   (5,1),  \!& (5,2),\!& (5,7),  \!&(6,0),  \!&(6,1),\!& (6,6),  \\
\!\!\!\!&\   (7,0),  \!& (7,5), \!& (8,0),  \!&(8,1),  \!&(8,2),\!& (8,4),  \\
\!\!\!\!&\  (8,5),  \!& (8,6), \!& (9,1),  \!&(9,3),  \!&(9,4),\!& (10,2), \\
\!\!\!\!&\  (10,5),  \!& (10,7),\!& (11,0),  \!&(11,3),  \!&(11,4),\!& (11,5),  \\
\!\!\!\!&\  (11,6),  \!& (12,4), \!&  (13,0),  \!&(14,1),  \!&(14,3),\!& (15,1),\\
\!\!\!\!&\  (15,2),  \!& (16,1),\!& (16,2),  \!&(16.3),  \!&(17,3),\!& (17,6), \\
\!\!\!\!&\  (18,0),  \!& (18,1),\!& (18,2),  \!&(19,0),  \!&(19,2),\!& (20,7), \\
\!\!\!\!&\  (21,0),  \!& (21,1), \!& (22,0),  \!&(22,1),  \!&(23,1),\!& (23,5), \\
\!\!\!\!&\  (25,2),  \!& (26,7),\!& (28,4),  \!&(18,6),  \!&(29,2),\!& (31,3), \\
\!\!\!\!&\  (32,2),  \!& (32,3), \!& (35,2),  \!&(36,1),  \!&(37,1),\!& (37,2),  \\
\!\!\!\!&\  (37,7),  \!& (38,0),\!& (43,1),  \!&(43,4),  \!&(44,0),\!& (44,1), \\
\!\!\!\!&\  (50,2),  \!& (51,6),\!&  (53,0),  \!&(53,3),  \!&(53,5),\!& (54,7), \\
\!\!\!\!&\  (58,2),  \!&(64,2),\!& (64,3), \!& (65,1),  \!& (65,4),\!& (72,2), \\
\!\!\!\!&\  (74,2), \!&(75,0),\!&(81,1),  \!& (85,5),  \!& (92,4),\!& (106,2), \\  
\!\!\!\!&\  (106,3), \!&(110,2),\!&(116,3),  \!&(123,1),  \!&(128,2),\!& (128,3)\} \\\end{array}
\end{equation}
and 
\begin{equation}\label{(a,r)neg}
\begin{array} {lllllll}
S^-:=&\{(1,7),  \!&(2,6),  \!&(2,7),\!& (4,4), \!& (4,5),  \!& (4,6),  \\
\!\!\!\!&\  (4,7),  \!&(7,7),  \!&(8,3),\!& (8,7), \!& (9,6),  \!& (10,0),  \\
\!\!\!\!&\  (11,2),  \!&(11,7),  \!&(15,5),\!& (18,6), \!& (18,7),  \!& (20,0),  \\
\!\!\!\!&\  (22,7),  \!&(23,6),  \!&(26,0),\!& (28,1), \!& (31,4),  \!& (36,6),  \\
\!\!\!\!&\  (37,0),  \!&(37,5),  \!&(44,6),\!& (44,7), \!& (53,2),  \!& (53,4),  \\
\!\!\!\!&\  (53,7),  \!&(54,0), \!& (106,4),  \!& (110,5),\!&(116,4), \!&(128,4),  \\
\!\!\!\!&\  (128,5)\}.  \!&   \!&  \!& \!& \!& \\
\end{array}
\end{equation}
Each pair $(A,r_B) \in S^+$ has integer solutions for only (\ref{pos}) and not for (\ref{neg}) and each pair $(A,r_B) \in S^-$ has integer solutions for only (\ref{neg}) and not for (\ref{pos}).
And then, we may easily see that for each 
$$1 \le A(m-2)+B \le 160(m-2)$$ 
with $(A,r_B) \notin S^+ \cup S^-$ where $r_B$ is the residue of $B$ modulo $8$ one of the above integer solutions $(x_1,x_2,x_3) \in \z^3$ for (\ref{pos}) or (\ref{neg}) satisfies
\begin{equation}\label{124;ab}\begin{cases} A(m-2)+B\equiv P_m(x_1)+2P_m(x_2)+4P_m(x_3) \pmod{8} \\
0 \le A(m-2)+B-\{ P_m(x_1)+2P_m(x_2)+4P_m(x_3)\} < m-2.
\end{cases}\end{equation}
Denote an integer solution for (\ref{124;ab}) by $(x_1(A,B),x_2(A,B),x_3(A,B)) \in \z^3$.
Since the remaining $$8P_m(x_4)+\cdots +2^{n-1}P_m(x_n)$$ would represents $$A(m-2)+B-\{ P_m(x_1(A,B))+2P_m(x_2(A,B))+4P_m(x_3(A,B))\}$$(which is a multiple of $8$ in $[0,m-3]$ from (\ref{124;ab})) by taking $P_m(x_i) \in \{0,1\}$ for all $4 \le i \le n$, we may obtain that every integer $A(m-2)+B$ in $[1, 160(m-2)]$ with $(A,r_B) \notin S^+ \cup S^-$ is represented by $$F_m(\mathbf x)=P_m(x_1)+2P_m(x_2)+\cdots+2^{n-1}P_m(x_n)$$

To complete the proof, we finally consider the integers $$A(m-2)+B \in [1,160(m-2)]$$ where $(A,r_B) \in S^+ \cup S^-$.
Among the pairs $(A,r_B)$ in $S^+ \cup S^-$, consider $(116,3) \in S^+$.
From
\begin{equation}\label{116,3}P_m(1)+2P_m(-7)+4P_m(6)=116(m-2)+11,\end{equation}
we may yield that every positive integer $A(m-2)+B \in [1,160(m-2)]$ with $(A,r_B)=(116,3)$ may be represented by $F_m(\mathbf x)$ other than $116(m-2)+3$.
On the other hand, $116(m-2)+3$ may be equivalent with one $115(m-2)+B'$ of the 
$$\begin{cases}
115(m-2)+32=P_m(-2)+2P_m(1)+4P_m(8)\\
115(m-2)+41=P_m(3)+2P_m(9)+4P_m(5)\\
115(m-2)+42=P_m(2)+2P_m(6)+4P_m(7)\\
115(m-2)+35=P_m(3)+2P_m(0)+4P_m(8)\\
115(m-2)+36=P_m(-2)+2P_m(9)+4P_m(5)\\
115(m-2)+37=P_m(11)+2P_m(1)+4P_m(6)\\
115(m-2)+38=P_m(2)+2P_m(10)+4P_m(4)\\
115(m-2)+39=P_m(-1)+2P_m(6)+4P_m(7)\\
\end{cases}$$
modulo $8$ because $\{32, 41,42,35,36, 37, 38, 39\}$ form a complete set of residues modulo $8$. 
By agian using the fact that the remaining $8P_m(x_4)+\cdots +2^{n-1}P_m(x_n)$ represents all the multiples of $8$ up to $m-3$, we may obtaint that $116(m-2)+3$ is also represented by $F_m(\mathbf x)$ because $$(116(m-2)+3)-(115(m-2)+B') $$ would be a multiple of $8$ in $[0,m-3]$.
Similarly with the above case $(A,r_B)=(116,3) \in S^+\cup S^-$, through surveys based on the result of integer solutions for (\ref{pos}) and (\ref{neg}) with the fact that the remaining $m$-gonal subform $8P_m(x_4)+\cdots+2^{n-1}P_m(x_n)$ of $F_m(\mathbf x)$ represents all the multiples of $8$ up to $m-3(\le 2^n-8)$
, we may get that every integer $A(m-2)+B$ in $[1,160(m-2)]$ with $(A,r_B) \in S^+ \cup S^-$ other than 
\begin{equation}\label{TT}\begin{array} {lll}
\!\!\!\! 37(m-2)+7, \!&37(m-2)+15, \!&37(m-2)+m-18,  \\
\!\!\!\!37(m-2)+m-16, \!&37(m-2)+m-11, \!&37(m-2)+m-10 \\  
\end{array}\end{equation}
may be represented by $F_m(\mathbf x)$ by taking $(x_4,\cdots,x_n) \in \{0,1\}^{n-3}$.
In the meanwhile, we may directly check that the representability of integers in (\ref{TT}) by $F_m(\mathbf x)$ as follows
$$\begin{cases}
37(m-2)+7= P_m(-1)+2P_m(4)+4P_m(0)+8P_m(-2)+16P_m(1)\\
37(m-2)+15= P_m(-1)+2P_m(0)+4P_m(-2)+8P_m(3)\\
37(m-2)+m-18= P_m(4)+2P_m(0)+4P_m(-3)+8P_m(-1)\\
37(m-2)+m-16= P_m(4)+2P_m(1)+4P_m(-3)+8P_m(-1)\\
37(m-2)+m-11= P_m(-3)+2P_m(-3)+4P_m(-2)+8P_m(-1)+16P_m(1)\\
37(m-2)+m-10= P_m(-4)+2P_m(0)+4P_m(-1)+8P_m(-2)+16P_m(1).\\
\end{cases}$$

Consequently, we may conclude that $$F_m(\mathbf x)=P_m(x_1)+2P_m(x_2)+\cdots+2^{n-1}P_m(x_n)$$ represents every positive integer up to $C(m-2)$, yielding the $F_m(\mathbf x)$ is universal by the Theorem \ref{C}.
\end{proof}

\vskip 0.5em

\begin{rmk} \label{rmk(rm)1}
In Lemma \ref{rm1}, we proved that for sufficiently large $m$ with $2 \le 2^{\ceil{\log_2(m-3)}}-m$, $$r_m=\ceil{\log_2(m-3)}$$
by showing that $$P_m(x_1)+2P_m(x_2)+4P_m(x_3)+\cdots +2^{n-1}P_m(x_n)$$ where $n=\ceil{\log_2(m-3)}$ is universal.
Since $-3 \le 2^{\ceil{\log_2(m-3)}}-m$, in order to complete the proof of Theorem \ref{min} now we need to consider $r_m$ for $-3 \le 2^{\ceil{\log_2(m-3)}}-m\le 4$.

\begin{itemize}
\item[(1)]
When $-3 = 2^{\ceil{\log_2(m-3)}}-m$, there is only one $m$-gonal form which represents every positive integer up to $m-4$ with the $\rank n= \ceil{\log_2(m-3)}$ that is $$P_m(x_1)+2P_m(x_2)+\cdots+2^{n-1}P_m(x_n)$$
which does not represent $m-2$, i.e., there is no $m$-gonal form which represents every positive integer up to $m-2$ of the $\rank$ less than or equal to $\ceil{\log_2(m-3)}$.
Clearly we may have that there is no universal $m$-gonal form of the $\rank$ less than or equal to ${\ceil{\log_2(m-3)}}$ when $-3 = 2^{\ceil{\log_2(m-3)}}-m$ which implies that $${\ceil{\log_2(m-3)}}+1 \le r_m.$$

\item[(2)]
When $-2 = 2^{\ceil{\log_2(m-3)}}-m$, all of the $m$-gonal forms which represent every positive integer up to $m-4$ with the $\rank n= \ceil{\log_2(m-3)}$ are 
$$\begin{cases}
P_m(x_1)+2P_m(x_2)+\cdots+2^{n-1}P_m(x_n) \text{ and }\\ 
$$P_m(x_1)+2P_m(x_2)+\cdots+(2^{n-1}-1)P_m(x_n)\\
\end{cases}$$
which do not represent $m-2$.
So in this case too, we may have that $${\ceil{\log_2(m-3)}}+1 \le r_m.$$


\item[(3)]
When $-1 = 2^{\ceil{\log_2(m-3)}}-m$, all of the $m$-gonal forms which represent every positive integer up to $m-4$ with the $\rank n= \ceil{\log_2(m-3)}$ are 
$$\begin{cases}
P_m(x_1)+2P_m(x_2)+\cdots+2^{n-1}P_m(x_n)\\ 
P_m(x_1)+2P_m(x_2)+\cdots+(2^{n-1}-1)P_m(x_n)\\ 
P_m(x_1)+2P_m(x_2)+\cdots+(2^{n-1}-2)P_m(x_n)\\
\end{cases}$$
which do not represent $2m-4$, $m-2$, and $m-2$, respectively which yields that $${\ceil{\log_2(m-3)}}+1 \le r_m.$$

\item[(4)]
When $0 = 2^{\ceil{\log_2(m-3)}}-m$, the all of $m$-gonal forms which represent every positive integer up to $m-4$ with the $\rank n= \ceil{\log_2(m-3)}$ are 
$$\begin{cases}
P_m(x_1)+2P_m(x_2)+\cdots+2^{n-2}P_m(x_{n-1})+2^{n-1}P_m(x_n)\\
P_m(x_1)+2P_m(x_2)+\cdots+2^{n-2}P_m(x_{n-1})+(2^{n-1}-1)P_m(x_n)\\
P_m(x_1)+2P_m(x_2)+\cdots+2^{n-2}P_m(x_{n-1})+(2^{n-1}-2)P_m(x_n)\\
P_m(x_1)+2P_m(x_2)+\cdots+2^{n-2}P_m(x_{n-1})+(2^{n-1}-3)P_m(x_n)\\
P_m(x_1)+2P_m(x_2)+\cdots+(2^{n-2}-1)P_m(x_{n-1})+(2^{n-1}-2)P_m(x_n)
\end{cases}$$
which do not represent $2m-3$, $2m-4$, $m-2$, $m-2$, and $m-2$, respectively which yields that $${\ceil{\log_2(m-3)}}+1 \le r_m.$$

\item[(5)]
When $1 = 2^{\ceil{\log_2(m-3)}}-m$, all of the $m$-gonal forms which represent every positive integer up to $m-4$ with the $\rank n= \ceil{\log_2(m-3)}$ are 
$$\begin{cases}
P_m(x_1)+2P_m(x_2)+\cdots+2^{n-2}P_m(x_{n-1})+2^{n-1}P_m(x_n)\\
P_m(x_1)+2P_m(x_2)+\cdots+2^{n-2}P_m(x_{n-1})+(2^{n-1}-1)P_m(x_n)\\
P_m(x_1)+2P_m(x_2)+\cdots+2^{n-2}P_m(x_{n-1})+(2^{n-1}-2)P_m(x_n) \\
$$ P_m(x_1)+2P_m(x_2)+\cdots+2^{n-2}P_m(x_{n-1})+(2^{n-1}-3)P_m(x_n)\\ 
$$ P_m(x_1)+2P_m(x_2)+\cdots+2^{n-2}P_m(x_{n-1})+(2^{n-1}-4)P_m(x_n)\\
$$ P_m(x_1)+2P_m(x_2)+\cdots+(2^{n-2}-1)P_m(x_{n-1})+(2^{n-1}-2)P_m(x_n)\\
$$ P_m(x_1)+2P_m(x_2)+\cdots+(2^{n-2}-1)P_m(x_{n-1})+(2^{n-1}-3)P_m(x_n)\\
\end{cases}$$
which do not represent $5(m-2)-1$, $2m-3$, $2m-4$, $m-2$, $m-2$, $m-2$, and $m-2$, respectively which yields that $${\ceil{\log_2(m-3)}}+1 \le r_m.$$
\item[(6)]
When $2\le  2^{\ceil{\log_2(m-3)}}-m \le 4$, one may see that $$P_m(x_1)+2P_m(x_2)+\cdots +2^{n-1}P_m(x_n)$$ where $n=\ceil{\log_2(m-3)}$ is universal by showing that the $m$-gonal form represents every positive integer up to $C(m-2)$ through similar processings with the Lemma \ref{rm1}.
But it would require more delicate care to examine the representability of small integers in $[1,160(m-2)]$ than the Lemma \ref{rm1} because of the tighter condition of integers represented by $8P_m(x_4)+\cdots +2^{n-1}P_m(x_n)$.
We omit the proof in this paper.
\end{itemize}
\end{rmk}

\vskip 0.5em

\begin{lem} \label{rm2}
For $m>2\left(\left(2C+\frac{1}{4}\right)^{\frac{1}{4}}+\sqrt{2}\right)^2$, the $m$-gonal form 
\begin{equation}\label{minm'}
F_m(\mathbf x)=P_m(x_1)+2P_m(x_2)+\cdots+2^{n}P_m(x_{n+1})
\end{equation}
where $n=\ceil{\log_2(m-3)}$ is universal.
\end{lem}
\begin{proof}
One may prove this lemma through almost same arguments with the proof of Lemma \ref{rm1}.
Actually, under the assumption of this lemma, the fact that the subform $$8P_m(\mathbf x_4)+\cdots +2^nP_m(\mathbf x_{n+1})$$ of $F_m(\mathbf x)$ represents all the multiples of $8$ up to $2m-14$ make to show this lemma easier than to show Lemma \ref{rm1}.
\end{proof}

\vskip 0.5em

\begin{rmk} \label{rmk(rm)2}
The Lemma \ref{rm2} says that $F_m(\mathbf x)$ is a universal $m$-gonal form of the $\rank \ceil{\log_2(m-3)}+1$.
On the other hand, we observed that there is no universal $m$-gonal form of the $\rank$ less than or equal to $\ceil{\log_2(m-3)}$ when $-3 \le 2^{\ceil{\log_2(m-3)}}-m\le 1$ in Remark \ref{rmk(rm)1}.
Therefore we may obtain that the $F_m(\mathbf x)$ in the Lemma \ref{rm2} is a universal $m$-gonal form of the minimal $\rank$ when $-3 \le 2^{\ceil{\log_2(m-3)}}-m\le 1$ which yields that $r_m=\ceil{\log_2(m-3)}+1$.
With Lemma \ref{rm1}, Remark \ref{rmk(rm)1}, and Lemma \ref{rm2}, we may claim Theorem \ref{min}.
\end{rmk}

\section{The maximal rank for a leaf of the escalator tree}

Throughout this section, we prove Theorem \ref{main}.
\begin{prop}\label{node}
A node $F_m(\mathbf x)=\sum_{i=1}^ka_iP_m(x_i)$ of the escalator tree represents every positive integer up to $\sum _{i=1}^ka_i$.
\end{prop}
\begin{proof} 
The proof proceeds by induction on  the rank $k$ of node.

When $k=1$, it is clear because $a_1=1$.

And now assume that the Proposition is true for all nodes of $\rank \ k-1$, that is, any node $\sum_{i=1}^{k-1}a_iP_m(x_i)$ of $\rank \  k-1$ represents every positive integer up to $\sum _{i=1}^{k-1}a_i$.
For a node $\sum _{i=1}^ka_iP_m(x_i)$ of the $\rank \ k$ to obtain a contradiction, assume that there is an integer $\alpha \le \sum _{i=1}^ka_i$ which is not represented by the node $\sum _{i=1}^ka_iP_m(x_i)$.
Since the truant of $\sum_{i=1}^{k-1}a_iP_m(x_i)$ is less than the truant of $\sum_{i=1}^{k}a_iP_m(x_i)$ (which is less than or equal to $\alpha$), we may get that $a_k \le \alpha$ because $a_k$ must be less than or equal to the truant of $\sum_{i=1}^{k-1}a_iP_m(x_i)$.
Since $0 \le \alpha-a_k \le \sum_{i=1}^{k-1}a_i$, by the induction hypothesis, $\alpha - a_k$ may be represented by $\sum _{i=1}^{k-1}a_iP_m(x_i)$.
Therefore $$\alpha=(\alpha - a_k)+a_k$$ may be represented by $\sum _{i=1}^ka_iP_m(x_i)$ by taking $x_k=1$, which is a contradiction.
This completes the proof.
\end{proof}

\vskip 0.5em

\begin{rmk} \label{key}
We may obtain that a node $$\sum \limits_{i=1}^ka_iP_m(x_i)$$ with $C(m-2) \le a_1+\cdots +a_k$ would be a leaf, i.e., a proper universal $m$-gonal form by Theorem \ref{C} and Proposition \ref{node}
\end{rmk}

\vskip 0.5em

\begin{lem} \label{2}
For $m>6C^2(C+1)$, a leaf $F_m(\mathbf x)=\sum \limits_{i=1}^na_iP_m(x_i)$ with $a_{l_m} \ge C+1$ where $l_m:=\floor {\frac{m-2}{C+1}}$ has the $\rank  \ n \le \left(1-\frac{1}{(C+1)^2} \right)(m-2)+\frac{C+2}{C+1}$.
\end{lem}
\begin{proof}
To obtain a contradiction assume that $n >(1-\frac{1}{(C+1)^2})(m-2)+\frac{C+2}{C+1}$.  
Then we may get 
\begin{align*}\sum \limits_{i=1}^{n-1}a_i & =\sum \limits_{i=1}^{l_m}a_i+\sum \limits_{i={l_m}+1}^{n-1}a_i \\
&>\left (\frac{m-2}{C+1}-1\right)+(C+1)\left(n-1-\frac{m-2}{C+1}\right)>C(m-2)\end{align*}
which yields that the parent $\sum \limits_{i=1}^{n-1}a_iP_m(x_i)$ of $F_m(\mathbf x)$ is already universal by Remark \ref{key}, which is a contradiction.
Consequently, we may conclude that $n \le (1-\frac{1}{(C+1)^2})(m-2)+\frac{C+2}{C+1}$.
\end{proof}

\vskip 0.5em

\begin{rmk}
Now we consider a node $\sum \limits_{i=1}^na_iP_m(x_i)$ with $0 \neq a_{l_m} < C+1$ where $l_m:=\floor {\frac{m-2}{C+1}}$.
One may easily see that there appear 5 consecutively same coefficients between $C$-th coefficient $a_C$ and $5C$-th coefficient $a_{5C}$, i.e., there is $C\le t \le 5C-4$ for which $$a_t=a_{t+1}=\cdots =a_{t+4}$$
because there are $4C+1$ components between $a_CP_m(x_C)$ and $a_{5C}P_m(x_{5C})$ and $1\le a_C \le a_{C+1} \le \cdots \le a_{5C} \le a_{l_m}<C+1$. 
\end{rmk}

\vskip 0.5em

\begin{lem} \label{A(m-2)}
For $A\in \mathbb N$, the $m$-gonal form $$\sum \limits_{i=1}^5A\cdot P_m(x_i)$$ represents all the multiples of $A(m-2)$.
\end{lem}
\begin{proof}
See Lemma 2.2 in \cite{KP}.
\end{proof}

\vskip 0.5em

\begin{prop}\label{b}
For $m \ge 6C^2(C+1)$, let $F_m(\mathbf x)=\sum \limits_{i=1}^na_iP_m(x_i)$ be a leaf with $0 \not= a_{l_m} < C+1$ where $l_m:=\floor {\frac{m-2}{C+1}}$.
If there is $C \le t \le 5C-4$ for which $a_t=a_{t+1}=\cdots=a_{t+4}$, then we rearrange the coefficients of $F_m(\mathbf x)$ except the 5 consecutive coefficients $a_t,\cdots,a_{t+4}$ as follows
$$b_i:=\begin{cases}a_i & \text{ when } i<t\\ a_{i+5} & \text{ when } i\ge t.
\end{cases}$$
And then we have the followings.

\begin{itemize}
\item[(1)] For $i \le l_m$, the inequality $b_i\le b_1+\cdots +b_{i-1}+1$ always holds.
\item[(2)] If there is $l_m < i \le C\cdot l_m$ such that $b_i> b_1+\cdots +b_{i-1}+1$, then $n < m-4$.
\end{itemize}
\end{prop}
\begin{proof}
(1) For $i<t$, if $C \le b_1+\cdots +b_{i-1}$, then since $b_i=a_i \le a_{l_m} \le C$ from the assumption, we may get that  $b_{i}\le b_1+\cdots +b_{i-1}+1$. If $a_1+\cdots +a_{i-1}=b_1+\cdots +b_{i-1} <C<m-4$, then the truant of the node $a_1P_m(x_1)+\cdots +a_{i-1}P_m(x_{i-1})$ would be $a_1+\cdots +a_{i-1}+1(<m-3)$.
So we obtain that $b_i=a_i \le a_1+\cdots +a_{i-1}+1=b_1+\cdots +b_{i-1}+1$.

If there is $t\le i \le l_m$ such that $b_{i}>b_1+\cdots +b_{i-1}+1$,
then we may get that $$C \le t \le i=(i-1)+1 \le b_1+\cdots+b_{i-1}+1 < b_i \le a_{l_m}.$$ 
This is a contradiction to $a_{l_m} \le C$.
This yields the claim.

(2) To obtain a contradiction, assume that $n \ge m-4$.
If there is $l_m\left(=\floor{\frac{m-2}{C+1}}\right)<i\le C\cdot l_m$ such that $b_{i}>b_1+\cdots +b_{i-1}+1$, then we may have that $$a_{i+5}>a_1+\cdots +a_{i+4}+1-5A \ge i +5-5A > \frac{m-2}{C+1}-5A \ge (C+1)^2.$$
And then from 
\begin{align*}\sum \limits_{i=1}^{n-1}a_i & =\sum \limits_{i=1}^{C\cdot l_m+4}a_i+\sum \limits_{i=C\cdot l_m+5}^{n-1}a_i\\  &\ge (C\cdot l_m +4)+(C+1)^2(n-1-C\cdot l_m-4)\\
 &\ge (C\cdot l_m +4)+(C+1)^2(m-9-C\cdot l_m)>C(m-2),\end{align*}
we may obtain a contradiction that the parent $\sum_{i=1}^{n-1}a_iP_m(x_i)$ of $F_m(\mathbf x)$ is already universal by Remark \ref{key}.
Consequently, we may conclude that $n < m-4$.
\end{proof}

\vskip 0.5em

\begin{rmk}
In Lemma \ref{2} and Proposition \ref{b}, we showed that the $\rank$ of leaf under some conditions does not exceed $m-4$.
Actually, what the conditions have in common is that its escalating is not happening that slowly.
Under such the conditions, we reach to leaf node, i.e., a universal $m$-gonal form before $m-4$ escalating steps.
\end{rmk}

\begin{lem} \label{13}
For $m>6C^2(C+1)$, let $F_m(\mathbf x)=\sum \limits_{i=1}^na_iP_m(x_i)$ be a leaf with $0 \not= a_{l_m} < C+1$ where $l_m:=\floor {\frac{m-2}{C+1}}$.
Then $$(a_1,a_2)=(1,1)\text{ or } (a_1,a_2,a_3) \in \{(1,2,2),(1,2,3),(1,2,4)\}$$ 
and there is $C \le t \le 5C-4$ such that $$a_t=a_{t+1}=\cdots =a_{t+4}=:A.$$
\begin{itemize}
\item[(1)] When $(a_1,a_2)=(1,1)$, if there is above $C \le t \le 5C-4$ for which $a_t=\cdots=a_{t+4}=A>6$, then $n < m-4$.
\item[(2)] When $(a_1,a_2,a_3)=(1,2,2)$, if there is above $C \le t \le 5C-4$ for which $a_t=\cdots=a_{t+4}=A>12$, then $n < m-4$.
\item[(3)] When $(a_1,a_2,a_3)=(1,2,3)$, if there is above $C \le t \le 5C-4$ for which $a_t=\cdots=a_{t+4}=A>12$, then $n < m-4$.
\item[(4)] When $(a_1,a_2,a_3)=(1,2,4)$, $n <m-4$.
\end{itemize}
\end{lem}

\begin{proof}
Since the coefficients of $F_m(\mathbf x)$ follows the conditions (\ref{coe}), we may have that $$(a_1,a_2)=(1,1),\text{ or } (a_1,a_2,a_3) \in \{(1,2,2),(1,2,3),(1,2,4)\}.$$ 
We prove this lemma only for the case $(a_1,a_2)=(1,1)$ and omit the proof of (2), (3), and (4) in this paper. 
One may prove those through similar arguments with the proof of (1).

Under same notation as in Proposition \ref{b}, in virtue of Proposition \ref{b}, we could assume that 
\begin{equation}\label{bas}b_i \le b_1+\cdots +b_{i-1}+1\end{equation} for all $1\le i \le \min \{n-5, C\cdot l_m \}$.
Under the assumption (\ref{bas}), we prove that $n \le C\cdot l_m+5 \left( \approx \frac{C}{C+1}m\right)<m-4$.
To obtain a contradiction, assume that $n >C\cdot l_m+5$.
From \begin{align*}\sum \limits_{i=1}^{C\cdot l_m}b_i=\sum \limits_{i=1}^{5C}b_i+\sum \limits_{i=5C+1}^{C\cdot l_m}b_i &\ge  5C+A(C\cdot l_m-5C)  \\
& \ge 5C+AC(\frac{m-2}{C+1}-1)-5AC  \\
& > (A-1)(m-2)+4
\end{align*}
with (\ref{bas}), we may have that $$\sum _{i=1}^{C\cdot l_m}b_iP_m(x_i)$$ represents every positive integer up to $(A-1)(m-2)+4$ by taking $x_i \in \{0,1\}$ for all $i$.
For integer $N \in [(A-1)(m-2)+5,A(m-2)]$, since $$0 <N-\{(6(m-2)+4\} < (A-1)(m-2)$$ holds, $N-\{(6(m-2)+4\} $ may be written as 
\begin{equation}\label{11'}N-\{6(m-2)+4\}=\sum \limits_{i=1}^{C\cdot l_m}b_iP_m(N(x_i))\end{equation} for some $(N(x_i))_i \in \{0,1\} ^{C\cdot l_m}$ by the above argument.
Up to reordering, we may assume that $(N(x_1),N(x_2))=(0,0), (1,0),$ or $(1,1)$. 
Note that 
\begin{equation}\label{11}\begin{cases} P_m(4)+P_m(0)=6(m-2)+4 \\ P_m(4)+P_m(1)=6(m-2)+5 \\  P_m(3)+P_m(3)=6(m-2)+6. \end{cases}\end{equation}
And then with (\ref{11'}) and (\ref{11}), we may see that the integer $N \in [(A-1)(m-2)+5,A(m-2)]$ is written as follows

$$\begin{cases}
N=P_m(4)+P_m(0)+\sum \limits_{i=3}^{C\cdot l_m}b_iP_m(N(x_i)) & \text{ when }(N(x_1),N(x_2))=(0,0) \\
N=P_m(4)+P_m(1)+\sum \limits_{i=3}^{C\cdot l_m}b_iP_m(N(x_i)) & \text{ when }(N(x_1),N(x_2))=(1,0) \\
N=P_m(3)+P_m(3)+\sum \limits_{i=3}^{C\cdot l_m}b_iP_m(N(x_i)) & \text{ when }(N(x_1),N(x_2))=(1,1). \\
\end{cases}$$
As a result, we may obtain that $\sum_{i=1}^{C\cdot l_m}b_iP_m(x_i)$ represents every positive integer up to $A(m-2)$.
On the other hand, by Lemma \ref{A(m-2)}, all the multiples of $A(m-2)$ may be represented by $$\sum_{i=t}^{t+4}a_iP_m(x_i).$$ 
Consequently, we may conclude that $$\sum_{i=1}^{C\cdot l_m+5}a_iP_m(x_i)=\sum_{i=1}^{C\cdot l_m}b_iP_m(x_i)+\sum_{i=t}^{t+4}a_iP_m(x_i)$$
represents every positive integer, i.e., the parent of $F_m(\mathbf x)$ is already universal, yielding a contradiction to $F_m(\mathbf x)$ is a leaf.
This completes the proof of (1).

By using the below equations instead of (\ref{11}), one may show (2), (3), and (4) through similar arguments with the above.

(2) When $(a_1,a_2,a_3)=(1,2,2),$ one may use following equations\\
\begin{equation}\label{122}
\begin{cases}
P_m(0)+2P_m(-2)+2P_m(-2)=12m-32\\
P_m(-3)+2P_m(-2)+2P_m(0)=12m-31\\
P_m(-4)+2P_m(-1)+2P_m(0)=12m-30\\
P_m(-3)+2P_m(-2)+2P_m(1)=12m-29\\
P_m(0)+2P_m(-3)+2P_m(1)=12m-28\\
P_m(1)+2P_m(-3)+2P_m(1)=12m-27.\\
\end{cases}
\end{equation}

(3) When $(a_1,a_2,a_3)=(1,2,3), $ one may use following equations\\
\begin{equation}\label{123}
\begin{cases}
P_m(-2)+2P_m(-2)+3P_m(-1)=12m-33\\
P_m(-2)+2P_m(0)+3P_m(-2)=12m-32\\
P_m(-3)+2P_m(-2)+3P_m(0)=12m-31\\
P_m(0)+2P_m(-3)+3P_m(0)=12m-30\\
P_m(1)+2P_m(-3)+3P_m(0)=12m-29\\
P_m(3)+2P_m(-2)+3P_m(-1)=12m-28\\
P_m(3)+2P_m(0)+3P_m(-2)=12m-27.\\
\end{cases}
\end{equation}

(4) When $(a_1,a_2,a_3)=(1,2,4), $ one may use following equations (note that when $a_3=4$, $A\ge 4$)\\
\begin{equation}\label{124}
\begin{cases}
P_m(2)+2P_m(-1)+4P_m(0)=3m-6\\
P_m(-1)+2P_m(-1)+4P_m(1)=3m-5\\
P_m(-2)+2P_m(0)+4P_m(1)=3m-4\\
P_m(3)+2P_m(0)+4P_m(0)=3m-3\\
P_m(-2)+2P_m(1)+4P_m(1)=3m-2\\
P_m(3)+2P_m(1)+4P_m(0)=3m-1\\
P_m(2)+2P_m(2)+4P_m(0)=3m\\
P_m(3)+2P_m(0)+4P_m(1)=3m+1.\\
\end{cases}
\end{equation}
\end{proof}

\vskip 0.5em

\begin{rmk}
In Lemma \ref{13}, we showed that if a leaf $F_m(\mathbf x)=\sum \limits_{i=1}^na_iP_m(x_i)$ with $0 \neq a_{l_m} < C+1$ has the 5 consecutively same coefficients $A$ greater than $12$ between $C$-th component and $5C$-th component, then its $\rank$ could not exceed $m-4$.
Especially under the condition (\ref{bas}), the $\rank$ would be less than or equal to $C\cdot l_m+5 (\approx \frac{C}{C+1}m)$ (more stirict calculations would much reduce the upper bound for $\rank \ n$).

On the other hand, next lemma may help to consider the upper bound for $\rank$ of leaves $F_m(\mathbf x)=\sum \limits_{i=1}^na_iP_m(x_i)$ with $0 \neq a_{l_m} < C+1$ for which every 5 consecutively same coefficients appearing between $C$-th component and $5C$-th component is less than or equal to $12$.
\end{rmk}

\begin{lem} \label{k(A)+1}
For $m>6C^2(C+1)$, let $F_m(\mathbf x)=\sum \limits_{i=1}^na_iP_m(x_i)$ be a leaf with $0 \neq a_{l_m} < C+1$ where $l_m:=\floor {\frac{m-2}{C+1}}$.
Then there is $C \le t \le 5C-4$ such that $$a_t=a_{t+1}=\cdots =a_{t+4}=:A.$$
For $A \ge 2$, let $i(A)$ be the smallest index satisfying $A-1 \le a_1+\cdots+a_{i(A)}$.
\begin{itemize}
\item[(1)]
If $a_{i(A)+1} \le A-1$, then $n < m-4$.
\item[(2)]
If $A+1 \le a_{(C-A-1)l_m+5}$, then $n < m-4$.
\end{itemize}
\end{lem}
\begin{proof}
(1) Similarly with Proposition \ref{b}, we rearrange the coefficients of $F_m(\mathbf x)$ except 6 coefficients $a_{i(A)+1},a_t,\cdots,a_{t+4}$ as follows $$c_i:=\begin{cases}a_i & \text{ when } i \le i(A)\\ a_{i+1} & \text{ when } i(A)+1 \le i < t-1 \\ a_{i+6} & \text{ when }  t-1 \le i.
\end{cases}$$
Through similar arguments with the proof of Proposition \ref{b}, one may induce that if $c_i> c_1+\cdots +c_{i-1}+1$ for some $ i \le  C\cdot l_m$, then $n < m-4$.
From now on, we prove the lemma under the assumption \begin{equation}\label{cas}\text{$c_i\le c_1+\cdots +c_{i-1}+1$ for any $i \le C\cdot l_m$}.\end{equation} 
To obtain a contradiction assume that $m-4\le n$.
Through similar arguments with the proof of Lemma \ref{13}, one may obtain that $\sum_{i=1}^{C\cdot l_m}c_iP_m(x_i)$ represents every positive integer up to $(A-1)(m-2)$ by taking $x_i \in \{0,1\}$ for all $i$.
For an integer $N \in [(A-1)(m-2),A(m-2)]$, we may see that \begin{equation} \label{eq2}0<N-a_{i(A)+1}P_m(x_{i(A)+1})<(A-1)(m-2)\end{equation} holds for some $x_{i(A)+1} \in \{-1,2\}$.
From the above argument, since the $N-a_{i(A)+1}P_m(x_{i(A)+1})$ of (\ref{eq2}) is represented by $\sum _{i=1}^{C\cdot l_m}c_iP_m(x_i)$, we may yield that
 $N$ is represented by $\sum _{i=1}^{C\cdot l_m}c_iP_m(x_i)+a_{i(A)+1}P_m(x_{i(A)+1})$ by taking $x_{i(A)+1} \in \{-1,2\}$.
So we may obtain that $$\sum \limits _{i=1}^{C\cdot l_m}c_iP_m(x_i)+a_{i(A)+1}P_m(x_{i(A)+1})$$ represents every positive integer up to $A(m-2)$.

On the other hand, by Lemma \ref{A(m-2)}, all the multiples of $A(m-2)$ may be represented by $\sum_{i=t}^{t+4}a_iP_m(x_i)$. 
Finally, we may conclude that 
$$\sum \limits_{i=1}^{C\cdot l_m +6}a_iP_m(x_i)=a_{i(A)+1}P_m(x_{i(A)+1})+\sum \limits _{i=1}^{C\cdot l_m}c_iP_m(x_i)+\sum \limits_{i=t}^{t+4}a_iP_m(x_i)$$ $\left(C\cdot l_m +6 \approx \frac{C}{C+1}m\right)$ is already universal, yielding a contradiction to the fact that $F_m(\mathbf x)$ is a leaf.
This completes the proof of (1).

(2) In virtue of Proposition \ref{b}, under same notation as in the Proposition \ref{b}, we could assume that 
\begin{equation}\label{bas'}b_i \le b_1+\cdots +b_{i-1}+1\end{equation} for all $1\le i \le \min \{n-5, C\cdot l_m\}$.
Under the assumption (\ref{bas'}), we prove that $n \le C\cdot l_m+5<m-4$.
To obtain a contradiction, assume that $n >C\cdot l_m+5$.
With (\ref{bas'}), from 
\begin{align*}
\sum_{i=1}^{C\cdot l_m}b_i & =\sum_{i=1}^{5C}b_i+\sum_{i=5C+1}^{(C-A-1)\cdot l_m-1}b_i+\sum_{i=(C-A-1)l_m}^{C\cdot l_m}b_i\\
&\ge 5C+A((C-A-1)l_m-1-5C)+(A+1)(A+1)l_m\\
&=5C+A(C\cdot l_m-1)+(A+1)l_m\\
&>5C+A\left(C\cdot\frac{m-2}{C+1}-2\right)+(A+1)\left(\frac{m-2}{C+1}-1\right)\\
&=5C+A(m-2)-A\left(\frac{m-2}{C+1}+2\right)+(A+1)\left(\frac{m-2}{C+1}-1\right) \\
&>A(m-2),
\end{align*}
we may induce that $\sum_{i=1}^{C\cdot l_m}b_iP_m(x_i)$ represents every positive integer up to $A(m-2)$.
So we may conclude that $$\sum_{i=1}^{C\cdot l_m+5}a_iP_m(x_i)=\sum_{i=1}^{C\cdot l_m}b_iP_m(x_i)+\sum_{i=t}^{t+4}a_iP_m(x_i)$$
is already universal, yielding a contradiction since $\sum_{i=t}^{t+4}a_iP_m(x_i)$ represents all the multiples of $A(m-2)$ by Lemma \ref{A(m-2)}.
This completes the proof.
\end{proof}

\begin{table}
\caption{} \label{t2}
\begin{tabular}{|c|c|}
\hline  \rule[-2.4mm]{0mm}{7mm}
A             &   $(a_1,\cdots,a_{i(A)})$ \\
\hline \hline
$1$ &  \\
\hline
$2$ & $(1)$ \\
\hline
$3$ & $(1,1),(1,2)$ \\
\hline
$4$ & $(1,1,1),(1,1,2),(1,1,3),(1,2)$ \\
\hline
$5$ & $(1,1,1,1), (1,1,1,2), (1,1,1,3), (1,1,1,4), (1,2,2), (1,2,3), (1,2,4)$ \\
\hline
 & $(1,1,1,1,1), (1,1,1,1,2), (1,1,1,1,3), (1,1,1,1,4),(1,1,1,1,5),$ \\
$6$ & $(1,1,1,2), (1,1,1,3), (1,1,1,4),(1,1,2,2), (1,1,2,3), (1,1,2,4), (1,1,2,5),$\\
& $ (1,2,2), (1,2,3), (1,2,4)$ \\
\hline
$7$ & $(1,2,2,2), (1,2,2,3), (1,2,2,4), (1,2,2,5), (1,2,2,6),(1,2,3)$ \\
 \hline
$8$ & $(1,2,2,2),(1,2,2,3), (1,2,2,4), (1,2,2,5), (1,2,2,6), (1,2,3,3), (1,2,3,4), $\\
 & $ (1,2,3,5), (1,2,3,6), (1,2,3,7)$\\
\hline
 & $(1,2,2,2,2),(1,2,2,2,3),(1,2,2,2,4),(1,2,2,2,5),(1,2,2,2,6),(1,2,2,2,7), $\\
$9$&$(1,2,2,2,8),(1,2,2,3), (1,2,2,4), (1,2,2,5), (1,2,2,6), (1,2,3,3), (1,2,3,4), $\\
 & $ (1,2,3,5), (1,2,3,6), (1,2,3,7)$\\
\hline
& $(1,2,2,2,2),(1,2,2,2,3),(1,2,2,2,4),(1,2,2,2,5),(1,2,2,2,6),(1,2,2,2,7), $\\
$10$&$(1,2,2,2,8),(1,2,2,3,3),(1,2,2,3,4), (1,2,2,3,5), (1,2,2,3,6), (1,2,2,3,7), $\\
&$(1,2,2,3,8),(1,2,2,3,9),(1,2,2,4), (1,2,2,5), (1,2,2,6), (1,2,3,3), (1,2,3,4), $\\
 & $ (1,2,3,5), (1,2,3,6), (1,2,3,7)$\\
\hline
& $(1,2,2,2,2,2),(1,2,2,2,2,3),(1,2,2,2,2,4),(1,2,2,2,2,5),(1,2,2,2,2,6),$\\
&$(1,2,2,2,2,7),(1,2,2,2,2,8),(1,2,2,2,2,9),(1,2,2,2,2,10),$\\
&$(1,2,2,2,3),(1,2,2,2,4),(1,2,2,2,5),(1,2,2,2,6),(1,2,2,2,7), $\\
$11$&$(1,2,2,2,8),(1,2,2,3,3),(1,2,2,3,4), (1,2,2,3,5), (1,2,2,3,6), (1,2,2,3,7), $\\
&$(1,2,2,3,8),(1,2,2,3,9),(1,2,2,4,4),(1,2,2,4,5),(1,2,2,4,6),(1,2,2,4,7),$\\
&$(1,2,2,4,8),(1,2,2,4,9),(1,2,2,4,10),(1,2,2,5), (1,2,2,6),$\\
&$  (1,2,3,3,3),(1,2,3,3,4),(1,2,3,3,5),(1,2,3,3,6),(1,2,3,3,7),$\\
&$(1,2,3,3,8),(1,2,3,3,9),(1,2,3,3,10), (1,2,3,4),(1,2,3,5), (1,2,3,6), $\\
 & $ (1,2,3,7)$\\
\hline
& $(1,2,2,2,2,2),(1,2,2,2,2,3),(1,2,2,2,2,4),(1,2,2,2,2,5),(1,2,2,2,2,6),$\\
&$(1,2,2,2,2,7),(1,2,2,2,2,8),(1,2,2,2,2,9),(1,2,2,2,2,10)$\\
&$(1,2,2,2,3,3),(1,2,2,2,3,4),(1,2,2,2,3,5),(1,2,2,2,3,6),(1,2,2,2,3,7),$\\
&$(1,2,2,2,3,8),(1,2,2,2,3,9),(1,2,2,2,3,10),(1,2,2,2,3,11),(1,2,2,2,4),$\\
&$(1,2,2,2,5),(1,2,2,2,6),(1,2,2,2,7), (1,2,2,2,8),(1,2,2,3,3),(1,2,2,3,4),$\\
$12$&$(1,2,2,3,5), (1,2,2,3,6), (1,2,2,3,7),(1,2,2,3,8),(1,2,2,3,9),(1,2,2,4,4), $\\
&$(1,2,2,4,5),(1,2,2,4,6),(1,2,2,4,7),(1,2,2,4,8),(1,2,2,4,9),(1,2,2,4,10),$\\
&$(1,2,2,5,5),(1,2,2,5,6),(1,2,2,5,7),(1,2,2,5,8),(1,2,2,5,9),(1,2,2,5,10), $\\
&$ (1,2,2,5,11),(1,2,2,6), (1,2,3,3,3),(1,2,3,3,4),(1,2,3,3,5),(1,2,3,3,6),$\\
&$(1,2,3,3,7),(1,2,3,3,8),(1,2,3,3,9),(1,2,3,3,10), (1,2,3,4,4), $\\
&$(1,2,3,4,5),(1,2,3,4,6),(1,2,3,4,7),(1,2,3,4,8),(1,2,3,4,9),$\\
 & $(1,2,3,4,10),(1,2,3,4,11), (1,2,3,5), (1,2,3,6),(1,2,3,7)$\\
\hline
\end{tabular}
\end{table}

\vskip 0.5em

\begin{rmk}
In Section 4, until now, we showed that a great part of leaf has $\rank$ smaller than $m-4$, 
more precisely, every leaf other than the leaves $F_m(\mathbf x)=\sum \limits_{i=1}^na_iP_m(x_i)$
 with 
\begin{equation}\label{fin}\begin{cases}(A;a_1,\cdots,a_{i(A)}) \text{ in Table \ref{t2} and}  \\ a_{i}=A \ \text{ for all }i(A)<i \le \min \{n,(C-A-1)l_m\} \end{cases}\end{equation}
has the $\rank$ less than $m-4$. 
We lastly consider the $\rank$ of leaves $F_m(\mathbf x)$ with (\ref{fin}).
One may notice that the leaves $F_m(\mathbf x)$ with (\ref{fin}) are very slowly escalated proper universal $m$-gonal form.
\end{rmk}

\vskip 0.5em

\begin{lem} \label{AAAAA}
The $m$-gonal forms $$a_1P_m(x_1)+\cdots +a_{i(A)}P_m(x_{i(A)})+AP_m(x_{i(A)+1})+\cdots+AP_m(x_n)$$
where $n=\ceil{\left(1-\frac{1}{2A}\right)(C+1)l_m+\left(1-\frac{1}{2A}\right)(C+1)}+i(A)+5$ with $(A;a_1,\cdots,a_{i(A)})$ in Table \ref{t2} other than $ (1;\ ),(3;1,1),$ and $(3;1,2)$ are universal.
\end{lem}
\begin{proof}
Similarly with Proposition \ref{b}, we rearrange the coefficients except the $5$ coefficients $a_t,\cdots, a_{t+4}$ as
$$b_i:=\begin{cases}a_i & \text{ when } i\le i(A)\\ a_{i+5} & \text{ when } i\ge i(A)+1.
\end{cases}$$
From 
\begin{align*}
\sum_{i=1}^{n}b_i&=\sum_{i=1}^{i(A)}b_i+\sum_{i=i(A)+1}^{n-5}b_i\\
&  =\sum_{i=1}^{i(A)}b_i+\sum_{i=i(A)+1}^{n-5}A\\
& \ge \sum_{i=1}^{i(A)}b_i+\left(A-\frac{1}{2}\right)(m-2),
\end{align*}
we may see that the $m$-gonal form $\sum_{i=1}^{n-5}b_iP_m(x_i)$ represents every positive integer up to $\floor{\left(A-\frac{1}{2}\right)(m-2)}$.
On the other hand, one may directly check that for each $(A;a_1,\cdots,a_{i(A)})=(A;b_1,\cdots,b_{i(A)})$ in Table \ref{t2} other than $ (1;\ ),(3;1,1),$ and $(3;1,2)$, $$a_1P_m(x_1)+\cdots+a_{i(A)}P_m(x_{i(A)})=b_1P_m(x_1)+\cdots+b_{i(A)}P_m(x_{i(A)})$$ represents complete residues modulo $A$ in $\left[m-3,\floor{\left(A-\frac{1}{2}\right)(m-2)}\right]$.
For example, for $(A;a_1,\cdots, a_{i(A)})=(2;1)$, we have a complete system of residues
$$P_m(-1)=m-3, \quad P_m(2)=m$$ modulo $A=2$
and for $(A;a_1,\cdots, a_{i(A)})=(4;1,1,1)$, we have a complete system of residues
\begin{align*} P_m(-1)+P_m(0)+P_m(0)=m-3, \quad & P_m(-1)+P_m(1)+P_m(0)=m-2,\\
P_m(-1)+P_m(1)+P_m(1)=m-1, \quad  & P_m(2)+P_m(0)+P_m(0)=m\\\end{align*}
modulo $A=4$.
Without difficulty, one may check that for the other cases too by hand.
We omit the lengthy calcuations in this paper.
And then an integer $N \in \left[\floor{\left(A-\frac{1}{2}\right)(m-2)},A(m-2)\right]$ may be written as 
$$N=b_1P_m(N_1)+\cdots +b_{i(A)}P_m(N_{i(A)})+AP_m(N_{i(A)+1})+\cdots +AP_m(N_{n-5})$$
where $N':=b_1P_m(N_1)+\cdots +b_{i(A)}P_m(N_{i(A)})$ is an integer in $[m-3,(A-1)(m-2)]$ which is equivalent with $N$ modulo $A$ and $(N_{i(A)+1},\cdots , N_{n-5}) \in \{0,1\}^{n-i(A)-5}$ since $$0\le N-N'\le (A-1)(m-2)+1$$ is a multiple of $A$.
So we may get that $\sum_{i=1}^{n-5}b_iP_m(x_i)$ represents every positive integer up to $A(m-2)$.

By Lemma \ref{A(m-2)}, since $$AP_m(x_{i(A)+1})+\cdots +AP_m(x_{i(A)+5})$$ represents all the multiples of $A(m-2)$, we may conclude that $$\sum_{i=1}^na_iP_m(x_i)=AP_m(x_{i(A)+1})+\cdots +AP_m(x_{i(A)+5})+\sum_{i=1}^{n-5}b_iP_m(x_i)$$ is universal.
\end{proof}

\begin{lem}
Let $F_m(\mathbf x)=\sum \limits_{i=1}^na_iP_m(x_i)$ be a leaf with $0 \neq a_{l_m} < C+1$ where $l_m:=\floor {\frac{m-2}{C+1}}$.
If there is $ t \le 5C-4$ for which $$a_t=a_{t+1}=\cdots =a_{t+4}=:A.$$
with $A\neq 1,3$, then we have $n <m-4$.
\end{lem}
\begin{proof}
From Lemma \ref{b}, Lemma \ref{k(A)+1}, and Lemma \ref{AAAAA}, one may yield this.
\end{proof}

\vskip 0.5em

\begin{rmk}
More delicate care may reduce the upper bound for the $\rank$ of $F_m(\mathbf x):=a_1P_m(x_1)+\cdots +a_{i(A)}P_m(x_{i(A)})+AP_m(x_{i(A)+1})+\cdots+AP_m(x_n)$ $$n\approx \left(1-\frac{1}{2A}\right)(m-2)$$ which makes $$a_1P_m(x_1)+\cdots +a_{i(A)}P_m(x_{i(A)})+AP_m(x_{i(A)+1})+\cdots+AP_m(x_n)$$ universal in the above lemma.
Actually, following \cite{KP}, especially for $(A;a_1,\cdots, a_{i(A)})$ of the form $(A;1,\cdots, 1)$ with $A \neq 1,3$,  we may take the below $n$
\begin{equation}\label{n}
n=\begin{cases}
\floor{\frac{m}{2}} & \text{ when }A=2 \\
\ceil{\frac{m-2}{4}}+2 & \text{ when }A=4 \\
\ceil{\frac{m-3}{A}}+(A-2) & \text{ when } 5\le A \le 12 \\
\end{cases}
\end{equation}
instead of $n=\ceil{\left(1-\frac{1}{2A}\right)(C+1)l_m+\left(1-\frac{1}{2A}\right)(C+1)}+i(A)+5\approx \left(1-\frac{1}{2A}\right)(m-2)$ in Lemma \ref{AAAAA}
and the $n$ of (\ref{n}) would be optimal.
The authors guess that for the most $(A;a_1,\cdots, a_{i(A)})$ in Table \ref{t2} other than $(3;1,1),$ and $(3;1,2)$ the optimal $n$ 
which makes $$a_1P_m(x_1)+\cdots +a_{i(A)}P_m(x_{i(A)})+AP_m(x_{i(A)+1})+\cdots+AP_m(x_n)$$ universal would be close to $\frac{m}{A}$ but not all.

Lastly, we consider the $\rank$ of leaves $F_m(\mathbf x)=\sum \limits_{i=1}^na_iP_m(x_i)$ with 
$$(a_1,\cdots,a_{(C-2)l_m+5})=(1,\cdots,1)$$ or $$(a_1,\cdots, a_{\min \{n,(C-4)l_m+5\}})=(1,1,3.\cdots,3) \text{ or }(1,2,3,\cdots, 3).$$
Actually among such the above leaves, the maximal $\rank R_m$ of leaves of $m$-gonal form's escalator tree would appear. 
\end{rmk}

\begin{lem} \label{A=1}
Let $F_m(\mathbf x)=\sum \limits_{i=1}^na_iP_m(x_i)$ be a node of the escalator tree with 
\begin{equation}\label{A=1,cond'}a_1=a_2=\cdots =a_{(C-2)l_m+5}=1.\end{equation}
If \begin{equation}\label{A=1,cond} m-4 \le a_1+a_2+\cdots+a_n,\end{equation} then the node would be universal, i.e., the node would become a leaf of the tree.
\end{lem}
\begin{proof}
Note that the $m$-gonal form 
$$f_m(\mathbf x):=a_6P_m(x_6)+\cdots +a_nP_m(x_n)$$ represents every positive integer up to $m-3$ except at most $5$ integers by taking $P_m(x_6) \in \{0,1,m-3\}$ and $P_m(x_i) \in \{0,1\}$ for all $7 \le i \le n$. 
And the integers not represented by $f_m(\mathbf x)$ would be consecutive.
On the specific, 
if there is an integer not represented by $f_m(\mathbf x)$, then the integer always would be in $[(C-2)l_m+1,m-4]$ and such the situation would happen only when 
\begin{equation}\label{A=1case}\begin{cases}a_1+\cdots+a_n-5<m-4  & \text{ or } \\ a_n>a_1+\cdots +a_{n-1}+1-5. & \end{cases}\end{equation}
When $a_1+\cdots+a_n-5<m-4$, the consecutive integers not represented by $f_m(\mathbf x)$ in $[(C-2)l_m+1,m-4]$ would be $$a_1+\cdots + a_n-4,a_1+\cdots + a_n-3,\cdots, m-4$$ and when $a_n>a_1+\cdots +a_{n-1}+1-5$, the consecutive integers not represented by $f_m(\mathbf x)$ in $[(C-2)l_m+1,m-4]$ would be $$a_1+\cdots+a_{n-1}-4,a_1+\cdots+a_{n-1}-3,\cdots,a_n-1.$$

In the cases that $f_m(\mathbf x)$ represents every positive integer up to $m-3$, we may conclude that $F_m(\mathbf x)$ is universal by using the fact that $$a_1P_m(x_1)+\cdots+a_5P_m(x_5)=P_m(x_1)+\cdots+P_m(x_5)$$ represents all the multiples of $m-2$ from Lemma \ref{A(m-2)}.

For the other cases, let $$E_1< E_2(=E_1+1) < \cdots < E_s(=E_1+(s-1))$$ where $s \le 5$ be all of the positive integers which are not represented by $f_m(\mathbf x)$ in $[(C-2)l_m+1,m-4]$.
By using Lemma \ref{A(m-2)} again, we may yield that $F_m(\mathbf x)$ represents every positive integer which is not congruent to $E_1,E_2,\cdots, E_s$ modulo $m-2$.
On the other hand, one may observe that 
$E_1-1$ is written as $f_m(\mathbf x)=a_6P_m(x_6)+\cdots +a_nP_m(x_n)$ with $x_i=1$ for all $6 \le i \le n-1$ so from the representation of $E_1-1$ by changing $P_m(x_6)=1$ to $P_m(2)=m$, we may obtain $$E_1+(m-2)=P_m(2)+P_m(0)+P_m(0)+\cdots+P_m(x_n),$$
 by changing both of $P_m(x_6)=P_m(x_7)=1$ to $P_m(2)=m$, we may obtain 
$$E_{2}+2(m-2)=P_m(2)+P_m(2)+P_m(0)+\cdots+P_m(x_n),$$ and  by changing all of $P_m(x_6)=P_m(x_7)=P_m(x_8)=1$ to $P_m(2)=m$, we may obtain 
$$E_{3}+3(m-2)=P_m(2)+P_m(2)+P_m(2)+\cdots+P_m(x_n),$$
i.e., $E_1+(m-2), E_2+2(m-2)$ and $E_3+3(m-2)$ may be represented by $f_m(\mathbf x)$.
And $E_s+1$ may be written as $f_m(\mathbf x)=a_6P_m(x_6)+\cdots +a_nP_m(x_n)$ with $x_i=0$ for all $7 \le i \le n-1$ so from the representation of $E_s+1$ by changing $P_m(x_7)=0$ to $P_m(-1)=m-3$, we may obtain $$E_s+(m-2)=P_m(x_6)+P_m(-1)+P_m(0)+P_m(0)+\cdots+P_m(x_n)$$
and by changing both of $P_m(x_7)=P_m(x_8)=0$ to $P_m(-1)=m-3$, we may obtain 
$$E_{s-1}+2(m-2)=P_m(x_6)+P_m(-1)+P_m(-1)+P_m(0)+\cdots+P_m(x_n),$$
i.e., $E_s+(m-2)$ and $E_{s-1}+2(m-2)$ may be represented by $f_m(\mathbf x)$.
So from Lemma \ref{A(m-2)}, we may conclude that every positive integer except the below at most $9$ positive integers 
\begin{equation}\label{1,exc}
\begin{array}{lllll}
E_1 \quad \quad  & \quad  \quad E_2          & \quad \quad \ \cdots & \quad \quad E_{s-1} & \quad \quad E_s \\
 & E_2+(m-2) & \quad \quad \ \cdots & E_{s-1}+(m-2) & \\
 &  & E_3+2(m-2) &  & \\
\end{array}
\end{equation} where $s \le 5$ is represented by $F_m(\mathbf x)$.

On the other hand, the representability of the above integers in (\ref{1,exc}) by $F_m(\mathbf x)$ may be directly confirmed.
Since $a_1+\cdots +a_n \ge m-4$, by Proposition \ref{node}, the integers $$E_1,E_2\cdots, E_s$$ (smaller than $m-4$) are represented by the node $F_m(\mathbf x)$ by taking $x_i \in \{0,1\}$ for all $i$ and moreover, in this situation, we may additionally assume that $P_m(x_1)=P_m(x_2)=1$.
And then from a representation of $E_1,\cdots ,E_{s-2}$ by $F_m(\mathbf x)$ with $P_m(x_1)=1$ by changing $P_m(x_1)=1$ to $P_m(2)=m$, we may see that $$E_2+(m-2),\cdots ,E_{s-1}+(m-2)$$ are represented by $F_m(\mathbf x)$ and from a representation of $E_1$ by $F_m(\mathbf x)$ with $P_m(x_1)=P_m(x_2)=1$ by changing both of $P_m(x_1)=P_m(x_1)=1$ to $P_m(2)=m$, we may see that $$E_3+2(m-2)$$ is represented by $F_m(\mathbf x)$.
This completes the proof.
\end{proof}

\vskip 0.5em

\begin{rmk}\label{A=1rmk}
Following the Guy's argument \cite{G}, the total sum of all coefficients of a leaf $\sum \limits_{i=1}^na_iP_m(x_i)$ must exceed $m-4$, i.e., $a_1+\cdots +a_n \ge m-4$ because otherwise, the integers in $[1,m-4]$ could not all be represented by the (universal) leaf.
So the coefficient condition (\ref{A=1,cond}) in Lemma \ref{A=1} on the total sum of all coefficients 
$$a_1+\cdots +a_n \ge m-4$$is essential for any leaf $\sum _{i=1}^n a_iP_m(x_i)$. 
By Lemma \ref{A=1}, we may easily induce that the $\rank$ of a leaf $\sum _{i=1}^na_iP_m(\mathbf x)$ with (\ref{A=1,cond'}) does not exceed $m-4$ since $a_1+\cdots +a_n \ge n$ and (\ref{A=1,cond}) holds for $n=m-4$ only if $a_1= \cdots = a_{m-5}=1$.
Since the truant of the node $$P_m(x_1)+\cdots +P_m(x_{m-5})$$ is $m-4$, it would have its childeren 
\begin{equation}\label{A=1,leaf}P_m(x_1)+\cdots +P_m(x_{m-5})+a_{m-4}P_m(x_{m-4})\end{equation} where $1\le a_{m-4}\le m-4$ and that's all.
Conesequently, we may yield that a leaf $$a_1P_m(x_1)+\cdots +a_nP_m(x_n)$$ with (\ref{A=1,cond'}) has $n \le m-4$ and there are excatly $m-4$ leaves with (\ref{A=1,cond'}) of the $\rank n= m-4$ of (\ref{A=1,leaf}).

Until now, we showed that the $\rank$ of a leaf $\sum \limits_{i=1}^na_iP_m(x_i)$ otherwise 
\begin{equation}\label{3cond} a_3=a_4=\cdots =a_{(C-4)l_m+5}=3 \end{equation} 
does not exceed $m-4$.
In Lemma \ref{A=3}, we treat the leaves $\sum \limits_{i=1}^na_iP_m(x_i)$ with (\ref{3cond}).
\end{rmk}

\vskip 0.5em
\begin{lem} \label{A=3}
Let $F_m(\mathbf x)=\sum \limits_{i=1}^na_iP_m(x_i)$ be a node of the escalator tree with
$a_3=a_4=\cdots =a_{(C-4)l_m+5}=3$.
\begin{itemize}
\item[(1) ] If  $3(m-3)-1 \le a_1+a_2+\cdots+a_n$ with $a_i \equiv 0 \pmod{3}  \text{ for all $3\le i \le n$}$, then the node would be universal, i.e., the node would become a leaf of the tree.
\item[(2) ] If $a_{n} \nequiv 0 \pmod{3}$, then the node would be universal, i.e., the node would become a leaf of the tree.
\end{itemize}
\end{lem}
\begin{proof}
(1) For $$b_i:=\begin{cases}a_i & \text{ for }i=1,2\\ a_{i+5} & \text{ for } i \ge 3,\end{cases}$$ similarly with the proof of Lemma \ref{A=1}, we may get that the $m$-gonal form $f_m(\mathbf x):=\sum_{i=1}^{n-5}b_iP_m(x_i)$ may represent every positive integer up to $3(m-2)$ except at most 5 positive integers in $[3(C-2)l_m+3,3(m-3)-1]$ and these integers would have the same residue modulo $3$ and the congruent integers would be consecutive.
On the specific, the situation that there is an integer in $[3(C-2)l_m+3,3(m-3)-1]$ not represented by $f_m(\mathbf x)$ happens only when 
\begin{equation}\label{A=3case}\begin{cases}a_1+\cdots+a_n-15=b_1+\cdots+b_{n-5}<3(m-3)-1  & \text{ or } \\ a_n>a_1+\cdots +a_{n-1}+1-15. & \end{cases}\end{equation}

If $f_m(\mathbf x)$ represents every positive integer up to $3(m-2)$, then we may conclude that $F_m(\mathbf x)$ is universal from Lemma \ref{A(m-2)}.

Now assume that there is an integer in $[3(C-2)l_m+3,3(m-3)-1]$ not represented by $f_m(\mathbf x)$ and let  $$E_1 < E_2(=E_1+3) < \cdots < E_s(=E_1+3(s-1))$$ where $1\le s\le 5$ be all of the positive integers not represented by $f_m(\mathbf x)$ in $[1,3(m-3)-1]$.
Through similar arguments with the proof of Lemma \ref{A=1}, we may obtain that $F_m(\mathbf x)$ represents every positive integer except the below at most $9$ integers
\begin{equation}\label{3,exc}
\begin{array}{lllll}
E_1 \quad \quad  & \quad  \quad E_2          & \quad \quad \ \cdots & \quad \quad E_{s-1} & \quad \quad E_s \\
 & E_2+(m-2) & \quad \quad \ \cdots & E_{s-1}+(m-2) & \\
 &  & E_3+2(m-2) &  & \\
\end{array}
\end{equation} by using the fact that $a_3P_m(x_3)+\cdots +a_7P_m(x_7)$ represents all the multiples of $3(m-2)$ from Lemma \ref{A(m-2)}.

On the other hand, the representability of the above exception integers in (\ref{3,exc}) by $F_m(\mathbf x)$ may also be directly confirmed similarly with the argument of the proof of Lemma \ref{A=1}.

(2) In virtue of (1), we may have that $a_1+\cdots +a_{n-1}<3(m-3)-1$.
And then we may see that the truant of $\sum_{i=1}^{n-1}a_iP_m(x_i)$ would be one of 
$$a_1+\cdots +a_{n-1}+1, \  a_1+\cdots +a_{n-1}+2,  \text{ or }a_1+\cdots +a_{n-1}+3,$$
which implies that $$a_n\le a_1+\cdots +a_{n-1}+3<3(m-3)+2.$$
For $$b_i:=\begin{cases}a_i & \text{ for }i=1,2\\ a_{i+5} & \text{ for } i \ge 3,\end{cases}$$ similarly with the proof of Lemma \ref{A=1}, we may obtain that $f_m(\mathbf x):=\sum_{i=1}^{n-5}b_iP_m(x_i)$ represents every positive integer up to $3(m-2)$ except at most 5 positive integers in $[3(C-2)l_m+3,3(m-3)-1]$ and these integers would have the same residue modulo $3$ and the congruent integers would be consecutive. 
On the specific, the situation that there is an integer in $[3(C-2)l_m+3,3(m-3)-1]$ not represented by $f_m(\mathbf x)$ happen only when 
\begin{equation}\label{A=3case'}a_n>a_1+\cdots +a_{n-1}+1-15.\end{equation}

If $f_m(\mathbf x)$ represents every positive integer up to $3(m-2)$, then we may conclude that the $F_m(\mathbf x)$ is universal from Lemma \ref{A(m-2)}.

Now assume that there is an integer in $[3(C-2)l_m+3,3(m-3)-1]$ not represented by $f_m(\mathbf x)$ and let  
$$E_1 < E_2(=E_1+3) < \cdots < E_s(=E_1+3(s-1))$$ 
where $1\le s \le 5$ be all of the positive integers not represented by $f_m(\mathbf x)$ in $[3(C-2)l_m+3,3(m-3)-1]$.
Through similar arguments with the proof of Lemma \ref{A=1}, we may obtain that $F_m(\mathbf x)$ represents every positive integer except the below at most $9$ integers
\begin{equation}\label{3,exc'}
\begin{array}{lllll}
E_1 \quad \quad  & \quad  \quad E_2          & \quad \quad \ \cdots & \quad \quad E_{s-1} & \quad \quad E_s \\
 & E_2+(m-2) & \quad \quad \ \cdots & E_{s-1}+(m-2) & \\
 &  & E_3+2(m-2) &  & \\
\end{array}
\end{equation} in virtue of the fact that $a_3P_m(x_3)+\cdots +a_7P_m(x_7)$ represents all the multiples of $3(m-2)$ from Lemma \ref{A(m-2)}.

On the other hand, the representability of the above exception integers in (\ref{3,exc'}) by $F_m(\mathbf x)$ may also be directly confirmed similarly with the argument of the proof of Lemma \ref{A=1}.
\end{proof}

\begin{proof}[Proof of Theorem \ref{main}]
From Lemma \ref{A=3}, we may yield that
the $\rank$ of a leaf $\sum \limits_{i=1}^na_iP_m(x_i)$ with 
\begin{equation} \label{3cond'}a_3=a_4=\cdots =a_{(C-4)l_m+5}=3\end{equation}
would not exceed $m-2$ because $a_1+a_2+a_3+\cdots+a_n\ge 2+3(n-2)$ under (\ref{3cond'}).
Overall, with Remark \ref{A=1rmk}, we may get an upper bound $m-2$ for the $\rank$ of leaf of the $m$-gonal form's escalating tree.

On the other hand, we may have that the below nodes 
\begin{equation}\label{A=3,last}
\begin{cases}
P_m(x_1)+P_m(x_2)+\sum \limits_{i=3}^{m-3}3P_m(x_i) & \text{ when } m\equiv 0 \pmod{3}\\
P_m(x_1)+P_m(x_2)+\sum \limits_{i=3}^{m-3}3P_m(x_i) & \text{ when } m\equiv 1 \pmod{3}\\
P_m(x_1)+2P_m(x_2)+\sum \limits_{i=3}^{m-4}3P_m(x_i) & \text{ when } m\equiv 2 \pmod{3}
\end{cases}
\end{equation}
are not a leaf (i.e., not a universal form) with the truant $3m-10,3m-12$, and $3m-12,$ respectively and we may see that from Lemma \ref{A=3}, that's all of the nodes which are not leaves of the $\rank$ greater than or equal to $m-3, m-3,$ and $m-4$, respectively because 

\begin{equation}\label{A=3,last'}
\begin{cases}P_m(x_1)+2P_m(x_2)+\sum \limits _{i=3}^{\frac{2m-3}{3}}3P_m(x_i) & \text{ when } m\equiv 0 \pmod{3} \\
P_m(x_1)+2P_m(x_2)+\sum \limits _{i=3}^{\frac{2m-5}{3}}3P_m(x_i) & \text{ when } m\equiv 1 \pmod{3}  \\ 
P_m(x_1)+P_m(x_2)+\sum\limits _{i=3}^{\frac{2m-4}{3}}3P_m(x_i) & \text{ when } m\equiv 2 \pmod{3} 
\end{cases}
\end{equation}
 are universal.
This completes the proof of Theorem and also bring out Remark \ref{rmk R}. 
\end{proof}

\vskip 0.5em


\begin{rmk}
In this paper, we proved for $m>2\left(\left(2C+\frac{1}{4}\right)^{\frac{1}{4}}+\sqrt{2}\right)^2$,
$$r_m=\begin{cases}\ceil{\log_2(m-3)}+1& \text{ when } -3 \le 2^{\ceil{\log_2(m-3)}}-m\le 1\\
\ceil{\log_2(m-3)} & \text{ when } \ \quad 2 \le 2^{\ceil{\log_2(m-3)}}-m
\end{cases}
$$
and for $m>6C^2(C+1)$,
$$R_m=\begin{cases}m-2 & \text{ when } m\nequiv 2 \pmod{3}\\m-3 & \text{ when } m\equiv 2 \pmod{3}.\end{cases}$$
Especially, we showed that $$F_m^{(0)}(\mathbf x):=P_m(x_1)+2P_m(x_2)+\cdots +2^{r_m-1}P_m(x_{r_m})$$ is a leaf of the minimal $\rank r_m$,  i.e., an universal $m$-gonal form of the minimal $\rank r_m$. 

Now for $m>$ with $2 \le 2^{\ceil{\log_2(m-3)}}-m$, consider a sequence $\{F_m^{(j)}(\mathbf x)\}_{j=0}^{m-4-r_m}$ of leaves (proper universal $m$-gonal forms) which is inductively constructed from $F_m^{(0)}(\mathbf x)$ by splitting the first component $a_iP_m(x_i)$ with non-one coefficient $a_i(>1)$ into two components $P_m(\cdot)$ and $(a_i-1)P_m(\cdot)$, namely, 
from the $(j+1)$-th leaf of $\{F_m^{(j)}(\mathbf x)\}_{j=0}^{m-4-r_m}$ with $a_i > 1$
$$F_m^{(j)}(\mathbf x):=P_m(x_1)+\cdots +P_m(x_{i-1})+a_iP_m(x_i)+\cdots +a_{n}P_m(x_{n})$$
the next $(j+2)$-th leaf is constructed as
$$F_m^{(j+1)}(\mathbf x):=P_m(x_1)+\cdots +P_m(x_{i-1})+P_m(x_i)+(a_i-1)P_m(x_{i+1})+\cdots +a_{n}P_m(x_{n+1})$$
by rearranging variables.
By using the notation $\mathcal C_j:=[a_1,\cdots, a_n]$ for $F_m^{(j)}(\mathbf x)=\sum \limits_{i=1}^na_iP_m(x_i)$, we may describe the stream of coefficients of the sequence $\{F_m^{(j)}(\mathbf x)\}_{j=0}^{m-4-r_m}$ as follows

$$\begin{array}{lllllllll}
\mathcal C_0=&[1, & \ 2,  & \ \quad 4, & \ \quad 8, & 16, & \cdots, & 2^{r_m-2},& 2^{r_m-1}] \\
\mathcal C_1=&[1, &1,1,  & \ \quad 4, & \ \quad 8, & 16, & \cdots , & 2^{r_m-2},& 2^{r_m-1}]\\
\mathcal C_2=&[1, &1,1,  & \ \ \ 1,3, & \ \quad 8, & 16, & \cdots , & 2^{r_m-2},& 2^{r_m-1}]\\
\mathcal C_3=&[1, &1,1,  & \ 1,1,2, & \ \quad 8, & 16, & \cdots , & 2^{r_m-2},& 2^{r_m-1}]\\
\mathcal C_4=&[1, &1,1,  & 1,1,1,1, & \ \quad8, & 16, & \cdots , & 2^{r_m-2},& 2^{r_m-1}]\\
\mathcal C_5=&[1, &1,1,  & 1,1,1,1, & \ \ \ 1,7, & 16, & \cdots , & 2^{r_m-2},& 2^{r_m-1}]\\
\mathcal C_6=&[1, &1,1,  & 1,1,1,1, & \ 1,1,6, & 16, & \cdots , & 2^{r_m-2},& 2^{r_m-1}]\\
\mathcal C_7=&[1, &1,1,  & 1,1,1,1, & 1,1,1,5, & 16, & \cdots , & 2^{r_m-2},& 2^{r_m-1}]\\
\quad \vdots &   & &  & \  \ \quad \vdots  &  & & & \ \quad \quad  \quad  .\\
\end{array}$$
Then we may see $\rank F_m^{(j)}(\mathbf x)=r_m+j$ ultimately there appear leaves of $\rank$ from $r_m$ and to $m-4$ in the sequence $\{F_m^{(j)}(\mathbf x)\}_{j=0}^{m-4-r_m}$.

On the other hand for $m>$ with $-3 \le 2^{\ceil{\log_2(m-3)}}-m\le 1$, even though every $m$-gonal form of the inductively constructed sequence $\{ F_m^{(j)}(\mathbf x)\}_{j=0}^{m-4+r_m}$ as the same manner with the above is universal, but not all of them are proper universal, i.e., for some $F_m^{(j)}(\mathbf x)=\sum \limits_{i=1}^na_iP_m(x_i)$, its proper subform $$F_m^{(j)}(\mathbf x)-2^{r_m-1}P_m(x_n)$$ would also be universal and the subform would be indeed a leaf(i.e., proper univeral $m$-gonal form).
For the smallest $J$ for which $F_m^{(J)}(\mathbf x)$ is not a leaf, by replacing the $j$-th $m$-gonal form $F_m^{(j)}(\mathbf x)=\sum \limits_{i=1}^na_iP_m(x_i)$ for all where $j \ge J$ of the above sequence $\{F_m^{(j)}(\mathbf x)\}_{j=0}^{m-4-r_m}$ as its subform $$F_m^{(j)}(\mathbf x)-2^{r_m-1}P_m(x_n)$$
which is a leaf, we may get a sequence of leaf $\{F_m^{(j)}(\mathbf x)\}_{j=0}^{m-4-r_m}$.
And the $$\rank F_m^{(j)}(\mathbf x)=\begin{cases}r_m+j & \text{when } j<J \\ r_m+j-1 & \text{when } j \ge J. \end{cases}$$
So in this case, there appear leaves of $\rank$ from $r_m$ and to $m-5$ in the sequence $\{F_m^{(j)}(\mathbf x)\}_{j=0}^{m-4-r_m}$.
And $P_m(x_1)+\cdots +P_m(x_{m-4})$ is a leaf of the $\rank m-4$.

When $m \nequiv 2 \pmod{3}$, we may show that $$P_m(x_1)+P_m(x_2)+3P_m(x_3)+\cdots +3P_m(x_{r_m-2})+6P_m(x_{r_m-1})$$ is a leaf of $\rank r_m-1=m-3$ by using Lemma \ref{A=3}.

From the above arguments, we may see that there is a leaf of $\rank n$ for each $r_m \le n \le R_m$.
\end{rmk}


\end{document}